\documentclass[11pt]{article}

\usepackage{amsthm, amssymb, amsmath, amsfonts, url, enumerate}

\theoremstyle{theorem}
\newtheorem{theorem}{Theorem}[section]

\newtheorem{claim}[theorem]{Claim}
\newtheorem{proposition}[theorem]{Proposition}

\newtheorem{lemma}[theorem]{Lemma}

\theoremstyle{remark}
\newtheorem{remark}[theorem]{Remark}
\theoremstyle{definition}

\newcommand{\up}[1]{^{(#1)}}

\newenvironment{prf}{\noindent{\bf Proof.~}}{\(\qed\)}
\newcommand{\BPF}{\begin{prf}} 
\newcommand {\EPF}{\end{prf}}

\def\R{{\mathbb{R}}}
\def\C{{\mathbb{C}}}
\def\E{{\mathbb{E}}}

\def\Pr{\mathbb{P}}
\newcommand{\expon}{\textnormal{exp}}
\newcommand{\tensor}{\otimes}
\newcommand{\eps}{\epsilon}
\newcommand{\vect}{\textnormal{vec}}
\newcommand{\tr}{\textnormal{tr}}
\renewcommand{\P}{\mathbb{P}}

\usepackage[margin=0.7in]{geometry}

\usepackage{verbatim}
\usepackage{graphicx}
\usepackage{epsfig}
\usepackage{float}
\usepackage[all]{xy}
\usepackage{color}

\usepackage{fullpage}
\usepackage{appendix}
\usepackage{mathtools}
\usepackage{cases}

\usepackage{hyperref}

\renewcommand{\i}{\mathbf{i}}
\newcommand{\wt}{\widetilde}
\newcommand{\ov}{\overline}
\hypersetup{
    colorlinks,
    citecolor=blue,
    filecolor=blue,
    linkcolor=blue,
    urlcolor=blue
}
\graphicspath{{./figs/}}
\allowdisplaybreaks

\begin{document}

\title{A Matrix Expander Chernoff Bound}
\author{
	Ankit Garg\thanks{Microsoft Research New England,
	\texttt{garga@microsoft.com}.}
\and
Yin Tat Lee\thanks{University Washington, \texttt{yintat@uw.edu}.
This work was partially supported by NSF grant CCF-1740551.}
\and
Zhao Song\thanks{Harvard University \& UT-Austin, \texttt{zhaos@utexas.edu}. }
\and
Nikhil Srivastava\thanks{UC Berkeley, \texttt{nikhil@math.berkeley.edu}.
This work was partially supported by NSF grant CCF-1553751 and a Sloan Research Fellowship.}
}

\begin{titlepage}
\maketitle

\begin{abstract}

We prove a Chernoff-type bound for sums of matrix-valued random variables
sampled via a random walk on an expander, confirming a conjecture due to Wigderson
and Xiao \cite{wxretraction}.
Our proof is based on a new multi-matrix extension of the Golden-Thompson
inequality which improves in some ways the inequality in \cite{Sutter2017} and may be
of independent interest, as well as an adaptation of an argument for the scalar case due to Healy
\cite{healy08}. Secondarily, we also provide a generic reduction showing that
any concentration inequality for vector-valued martingales implies a
concentration inequality for the corresponding expander walk, with a weakening
of parameters proportional to the squared mixing time.

\end{abstract}
\thispagestyle{empty}
\end{titlepage}

\section{Introduction}

The Chernoff Bound \cite{chernoff} is one of the most widely used probabilistic results in
computer science. It states that a sum of independent bounded random
variables exhibits subgaussian concentration around its mean.  In particular,
when the random variables are i.i.d. samples from a fixed distribution, it implies
that the empirical mean of $k$ samples is  $\epsilon$-close to the true mean with exponentially
small deviation probability proportional to $e^{-\Omega(k\epsilon^2)}$. 

An important generalization of this bound was achieved by Gillman \cite{gillman}
(with refinements later by \cite{lezaud, kahale, LeonP04, wx05, healy08, Wagner08, ChungLLM12, RaoR17}), who
significantly relaxed the independence assumption to Markov dependence. In particular, suppose $G$ is a regular graph with vertex set $V=[n]$,
$X:V\rightarrow \C$ is a bounded function, and
$v_1,\ldots,v_k$ is a stationary random walk\footnote{That is the first vertex $v_1$ is chosen uniformly at random -- which is the stationary distribution of the graph $G$.} of length $k$ on $G$. Then, even
though the random variables $X(v_i)$ are in not independent (except when $G$ is the
complete graph with self loops), it is shown that:
\begin{equation}\label{eqn:expchernoff} \Pr \left[ \left|\frac1k\sum_{i=1}^k X(v_i) - \E [X ] \right|>\epsilon\right]\le
2 \cdot \exp(-\Omega((1-\lambda)k\epsilon^2)),\end{equation}
where $1-\lambda$ is the spectral gap of the transition matrix of the random
walk.  The gain here is that sampling a stationary random walk of length $k$ on
a constant degree graph with constant spectral gap requires $\log(n)+O(k)$ random bits, which is much
less than the $k\log(n)$ bits required to produce $k$ independent samples. Since
such graphs can be explicitly constructed,  this leads to a generic ``derandomization'' of the Chernoff
bound, which has had several important applications (see \cite{wx05} for a detailed discussion). 
In particular, it leads to the following randomness efficient sampler for scalar-valued functions (\cite{gillman}) using known strongly explicit constructions of expander graphs \cite{rvw2000, lps88}:

\begin{theorem}[\cite{gillman}] For any $\eps > 0$ and $k \ge 1$, there is a $\textnormal{poly}(r)$-time computable sampler $\sigma: \{0,1\}^r \rightarrow [n]^k$, where $r = \log(n) + O(k)$ s.t. for all functions $f: [n] \rightarrow [-1,1]$ satisfying $\E f=0$, we have that
$$
\Pr_{w \in_R \{0,1\}^r} \left[  \bigg|\frac{1}{k}\sum_{i=1}^k f(\sigma(w)_i) \bigg| \ge \eps \right] \le 2\exp\left(-\Omega\left( -\eps^2 k\right)\right)
$$ 
\end{theorem}
In many applications of interest $k$ is about $\log(n)$, and going from
$O(\log^2(n))$ to $O(\log(n))$ random bits leads to a complete derandomization
by cycling over all seeds $w\in\{0,1\}^r$.

A different generalization of the Chernoff bound appeared in the works of
Rudelson \cite{rudelson}, Ahlswede-Winter \cite{ahlswedewinter}, and Tropp
\cite{tropp}, who showed that a similar
concentration phenomenon is true for {\em matrix-valued} random variables. In particular, if
$X_1,\ldots,X_k$ are independent $d\times d$ complex Hermitian random matrices with $\|X_i\|\le 1$,
then the following is true: 
\begin{equation}\label{eqn:matrixchernoff} \Pr \left[ \left\|\frac1k\sum_{i=1}^k X_i - \E [ X ] \right\|>\epsilon\right]\le
2d\cdot \exp(-\Omega(k\epsilon^2)).\end{equation}
The only difference between this and the usual Chernoff bound is the factor of
$d$ in front of the deviation probability; to see that it is necessary, notice
that the diagonal case simply corresponds to a direct sum of $d$ 
arbitrarily correlated instances of the scalar Chernoff bound, so by the union
bound the probability should be $d$ times as large in the worst case. This so
called ``Matrix Chernoff Bound'' has seen several applications as well, notably
in quantum information theory, numerical linear algebra, and spectral graph
theory; the reader may consult e.g. the book \cite{tropp2015book} for many
examples. 

We present two different extensions of the above results in this paper.

\subsection{A Matrix Expander Chernoff Bound}
It is natural to wonder whether there is a common
generalization of \eqref{eqn:expchernoff} and \eqref{eqn:matrixchernoff}, i.e.,
a ``Matrix Expander Chernoff Bound''. Such a result was conjectured by Wigderson
and Xiao in \cite{wxretraction} --- in fact, \cite{wx05} contained a proof of it, but the authors
later discovered a gap in the proof. In this paper, we prove the Wigderson and Xiao conjecture, namely:
\begin{theorem}\label{thm:main} Let $G=(V,E)$ be a regular undirected graph whose transition matrix has second eigenvalue $\lambda$, 
	and let $f:V \rightarrow \C^{d \times d}$ be a function such that:
\begin{enumerate}
\item For each $v \in V$, $f(v)$ is Hermitian and $\|f(v)\| \le 1$. 
\item $\sum_{v \in V} f(v) = 0$. 
\end{enumerate}
 Then, for a stationary random walk $v_1,\ldots,v_k$ with $\epsilon \in (0,1)$ we have:
\begin{align*}
&\Pr\left[ \lambda_{\textnormal{max}} \left( \frac{1}{k}\sum_{j=1}^k f(v_j)\right) \ge ~~ 
\eps\right] \le d \cdot \expon\left(-\Omega\left( \eps^2 (1 - \lambda)
k \right) \right), \\
&\Pr\left[ \lambda_{\textnormal{min}} \left( \frac{1}{k} \sum_{j=1}^k f(v_j)\right) \le -\eps\right] \le d \cdot \expon\left(-\Omega\left( \eps^2 (1 - \lambda)
k \right) \right).
\end{align*} 
\end{theorem}

This theorem adds to the amazingly long list of pseudorandom properties of expander graphs. By applying the theorem with
a strongly explicit bounded degree expander, one obtains the following randomness-efficient sampler for matrix-calued functions conjectured in \cite{wxretraction}.
\begin{theorem}\label{thm:derand} For any $\eps > 0$, $k \ge 1$ and $d \ge 1$, there is a $\textnormal{poly}(r)$-time computable sampler $\sigma: \{0,1\}^r \rightarrow [n]^k$, where $r = \log(n) + O(k)$ s.t. for all functions $f: [n] \rightarrow \C^{d \times d}$ satisfying $\E f=0$ and for each $v \in [n]$, $f(v)$ is Hermitian and $\|f(v)\| \le 1$, we have that
$$
\Pr_{w \in_R \{0,1\}^r} \left[  \left\| \frac{1}{k}\sum_{i=1}^k f(\sigma(w)_i) \right\| \ge \eps \right] \le 2d \exp\left(-\Omega\left( -\eps^2 k\right)\right)
$$ 
\end{theorem}

We remark that while the derandomization applications studied in \cite{wx05}
were later recovered in \cite{WX08} using the method of pessimistic estimators,
that method requires additional assumptions to be efficiently implementable
(specifically, computability of the matrix moment generating function, which is
problem-dependent) and therefore does not constitute a truly black box
derandomization of the matrix Chernoff bound, whereas Theorem \ref{thm:derand}
does. Given the increasing ubiquity of applications of this bound, we therefore
suspect that it will find further applications in the study of derandomization
and expander graphs, beyond the ones mentioned in \cite{wx05}.

\subsubsection*{Techniques}
To describe the ideas that go into the proof of Theorem \ref{thm:main}, let us begin by recalling how the usual scalar
Chernoff bound is proved, in the case when the random variables have mean zero.
The key observation is that if $X_1,\ldots,X_k$ are
independent random variables, then the moment generating function of the sum is equal to the
product of the moment generating functions: 
$$
\E \left[ \exp\left(t\sum_{i=1}^k X_i \right) \right] =
\prod_{i=1}^k \E[ \exp(tX_i) ].
$$
This is no longer true in case where the $X_i$
come from a random walk, but we still have the algebraic fact that
\begin{equation}\label{eqn:commute} 
\exp\left(t\sum_{i=1}^k X_i \right) = \prod_{i=1}^k \exp(tX_i),
\end{equation} 
which allows one to decompose the sum as a product.
The latter allows one to consider the steps of the random walk separately and
analyze the change in the expectation inductively.

The analogue of the moment generating function in the matrix setting is 
$$\E \left[ \tr\left[ \exp \left(t\sum_{i=1}^k X_i \right) \right] \right],$$
and the main difficulty is that \eqref{eqn:commute} no longer holds if the
matrices $X_i$ do not commute. A substitute for this fact is given by the
Golden-Thompson inequality \cite{golden,thompson}, which states that for any
Hermitian $A,B$:
\begin{equation}\label{eqn:gt}
\tr[ \exp ( A + B ) ] \le \tr[ \exp(A) \exp(B) ] .
\end{equation}
The latter expression may further be bounded by
$\|\exp(A)\| \tr[ \exp(B) ]$, and this is sufficient to prove \eqref{eqn:matrixchernoff} in the independent case
as is done in \cite{ahlswedewinter},
where an inductive application of it yields 
$$ \E \left[ \tr\left[\exp\left(t\sum_{i=1}^k X_i \right) \right] \right] \le \tr[I] \cdot\prod_{i=1}^k \left\|\E [ \exp(tX_i) ] \right\|.$$
However, this approach is too crude to handle the Markov case, roughly because
in the absence of inependence, passing to the norm makes it difficult to utilize
the fact that the expectation of each $X_i$ is zero.

The original proof of Wigderson-Xiao was based on the following plausible
multi-matrix generalization of \eqref{eqn:gt}:
$$\tr\left[\exp\left(\sum_{i=1}^k A_i \right) \right]\le \tr\left[\prod_{i=1}^k \exp(A_i) \right],$$
which turns out to be false for $k>2$. To see why, observe that the left hand
side is always nonnegative, whereas the right hand side can be the trace of a
product of any three positive semidefinite matrices, which can be negative (and
this is not the case for two matrices). This led to a fatal gap in their proof.

The main ingredient in our proof is a new multi-matrix generalization of
\eqref{eqn:gt}, which is inspired by the following statement that was recently
proven in \cite{Sutter2017} (see also \cite{HKT16}).

\begin{theorem}[Corollary 3.3 in \cite{Sutter2017}]\label{thm:generalGT} Let $H_1,\ldots,H_k \in \C^{d \times d}$ be Hermitian matrices. Then
$$
\log \left[ \tr\left( \expon\left( \sum_{j=1}^k H_j \right)\right)\right] \le \int_{-\infty}^{\infty} \log\left[ \tr\left( \prod_{j = 1}^k \expon\left( \frac{H_j(1 + \i b)}{2}\right) \prod_{j = k}^1 \expon\left( \frac{H_j(1 - \i b)}{2}\right)\right)\right] \mathrm{d} \mu (b)
$$
where $\mu$ is some probability distribution on $(-\infty, \infty)$.
\end{theorem}

The above inequality successfully relates the matrix exponential of a sum to a
product of matrix exponentials, but is not adequate for proving an optimal
Chernoff bound. The reason is that all known arguments require a Taylor
expansion, and Theorem \ref{thm:generalGT} involves integration over an unbounded region (this
region can be truncated, but this introduces a loss which leads to a suboptimal bound). To remedy
this, we prove a new multi-matrix Golden-Thompson inequality, which only
involves integration over a bounded region instead of a line.
\begin{theorem}[Bounded Multi-matrix Golden-Thompson inequality]\label{thm:our_gt}
 Let $H_1,\ldots,H_k \in \C^{d \times d}$ be Hermitian matrices. Then
\begin{align*}
\log \left( \tr \left[ \exp \left( \sum_{j=1}^k H_j \right) \right] \right) \leq \frac{4}{\pi} \int_{-\frac{\pi}{2}}^{\frac{\pi}{2}} \log \left( \tr \left[ \prod_{j=1}^k \exp \left( \frac{e^{\i \phi}}{2} H_j \right) \prod_{j=k}^1 \exp \left( \frac{ e^{-\i \phi} }{2} H_j \right) \right] \right) \mathrm{d} \mu (\phi)
\end{align*}
where $\mu$ is some probability distribution on $[-\frac{\pi}{2},\frac{\pi}{2}]$.
\end{theorem}
We present the proof in Section~\ref{sec:gt}.  Theorem~\ref{thm:our_gt} is likely to be
of independent interest and could have further applications, e.g. in quantum information theory. We draw attention to the following notable features of the above two theorems:
\begin{enumerate}
\item [(a)] Since $\exp(H(1+ \i b))=\exp(H(1- \i b))^*$ and $\exp \left( \frac{e^{\i \phi}}{2} H_j \right) = \exp \left( \frac{ e^{-\i \phi} }{2} H_j \right)^*$ for Hermitian $H$, the right
hand always considers the trace of a matrix times its adjoint, which is always positive
semidefinite, ruling out the bad example described above. 
\item [(b)] They are {\em average case} inequalities, where the averaging is done over specific distributions. We remark
that the first inequality is known to be false in the worst case (i.e., with $b=0$ and with other small values of $b$;
see \cite{Sutter2017} for a discussion). 
		
\end{enumerate}





The main point is that Theorem \ref{thm:our_gt} allows one to relate the
exponential of a sum of matrices to a (two-sided) product of bounded $d\times
d$ matrices and their adjoints. In order to prove Theorem \ref{thm:main}, we
rewrite this as a one-sided matrix product of $d^2\times d^2$ matrices acting
on $\C^{d\times d},$ by encoding left and right multiplication on this space
via a tensor product. However, these matrices are no longer Hermitian (or even
normal), so it is difficult to analyze the moment generating function of their
product over the random walk using the perturbation-theoretic approach of
\cite{wx05}. We surmount this difficulty by employing a variant of the more
robust linear algebraic proof technique of Healy \cite{healy08}.  The proof of
Theorem \ref{thm:main} is presented in Section \ref{sec:mainproof}.

\subsection{Martingale Approximation of Expander Walks}
While Theorem \ref{thm:main} provides a satisfactory generalization of the expander
Chernoff bound to the case when one is interested in the spectral norm of a
matrix-valued function on $V$, one could ask what happens for other matrix norms
(such as Schatten norms), or even more generally, for functions taking values in
an arbitrary Banach space. Our second contribution is a generic reduction from
this problem, of proving concentration for random variables sampled using a
{Markov chain}, to the much more well-studied problem (see e.g.
\cite{chungliu}) of concentration
for sums of {\em martingale} random variables. 

\begin{theorem}\label{thm:martingale} Suppose $G=(V,E)$ is a regular graph whose transition matrix has second eigenvalue $\lambda$
	and
	$f:V\rightarrow \R^N$ is a vector-valued function satisfying $\sum_{v\in
	V}f(v)=0$ with $F:=\sqrt{\sum_{v\in V}\|f(v)\|_2^2}$, where $\|\cdot\|_2$ denotes the Frobenius norm.
	If $v_1,\ldots,v_k$ is a stationary random walk on $G$,
	then for every $\epsilon>0$, there is a martingale difference sequence $Z_1,\ldots, Z_k$ with
	respect to the filtration generated by initial segments of
	$v_1,\ldots,v_k$ such that
	\begin{align*}
	\frac 1k\sum_{i=1}^k f(v_i) = W+\frac1k \sum_{i=1}^k Z_i,
	\end{align*} where 
	\begin{enumerate}
	\item $W$ is a random vector satisfying $\|W\|_2\le \epsilon$.
	\item 
		Each term $Z_i$ satisfies 
		$$\|Z_i\|_*\le \frac{2\log(F/\epsilon)}{1-\lambda}\cdot \max_{v\in V}\|f(v)\|_*$$
		for every norm $\|\cdot\|_*$.
	\end{enumerate}
\end{theorem}

Thus, the empirical sums of any bounded (in any norm) function on a graph are
well-approximated by a martingale whose increments are also bounded, with a loss
in the bound depending on the $\ell_2$ norm $F$ of the function and the spectral gap of the
graph. Since $F$ will typically scale with the number of vertices, the ratio
above is typically comparable to the mixing time. 

To see the theorem in action, consider the case when $f(v)$ is matrix-valued
in $d\times d$ Hermitian matrices and $\|\cdot\|_*$ is the operator norm. If
$\|f(v)\| \le 1$ then 
we have the bound
$$F^2 =\sum_{v\in V}\|f(v)\|_F^2 \le dn.$$
Suppose we are interested in obtaining an estimate on the probability:
\begin{equation}\label{eqn:probmart}\P\left[\left\|\frac{1}{k}\sum_{v\in V}
f(v)\right\| >\epsilon\right].\end{equation}
Applying Theorem \ref{thm:martingale} with parameter $\epsilon/2$ and noting
that $\|W\| \le \|W\|_F$, we have that
\eqref{eqn:probmart} is at most 
$$ \P\left[\left\|\frac1k\sum_{i=1}^k Z_i\right\|>\epsilon/2\right],$$
where $Z_i$ is a martingale with bound
$$\|Z_i\| \le \frac{\log(nd/\epsilon)}{1-\lambda}.$$
We now appeal to the existing martingale generalization of \eqref{eqn:matrixchernoff} (see
e.g. \cite{tropp}) and find that this probability is at most
$$2d\cdot\exp\left(-\Omega\left(\frac{k\epsilon^2(1-\lambda)^2}{\log^2(nd)}\right)\right).$$

While this theorem is much weaker than the previous one in terms of parameters
(depending on the square of the mixing time rather than on the spectral gap), it
shows qualitatively that concentration for Markov chains is a generic
phenomenon rather than something specific to matrices. It also allows one to
instantly import the wealth of results regarding concentration for martingales
in various Banach spaces (see e.g., \cite{ledouxtalagrand}) to the random walk setting, albeit with suboptimal
parameters.  The simple proof of Theorem \ref{thm:martingale} is presented in Section \ref{sec:martingale}.

\section{Preliminaries}

For an $n \in {\mathbb{N}}_+$, let $[n]$ denote the set $\{1,2,\cdots,n\}$.
Let $\i$ denote $\sqrt{-1}$. For $z = a+ \i b$, where $a, b\in \R$, we define the complex conjugate of $z$ to be $\overline{z} = a - \i b$ and $|z| = \sqrt{a^2 + b^2}$. Then $|z|^2 = z \overline{z}$. We define real part $\text{Re}(z) = a$ and imaginary part $\text{Im(z)}=b$. Then $\text{Re}(z) = \frac{z+\overline{z}}{2}$ and $\text{Im}(z) = \frac{z-\overline{z}}{2\i}$.

We will be working with $D$-regular undirected graphs $G = (V,E)$. The number of
vertices of the graph, $V$ will be denoted by $n$. $A$ will denote the
adjacency matrix of the graph and $P = A/D$ will denote its normalized adjacency
matrix. A regular graph $G$ will be called a $\lambda$-expander ($0 < \lambda <
1$) if $\|P x\| \le \lambda \cdot \| x \|$ for all vectors
$ x \in \C^n$ s.t. $\sum_{i=1}^n  x_i = 0$. 

We will use $e_i \in \C^d$ to denote the standard basis vector with $1$ in
$i^{\text{th}}$ position and $0$ everywhere else. 

\subsection{Linear Algebra}
\noindent {\bf Matrices and Norms.} For matrix $A$, we use $A^\top$ to denote the transpose of $A$, we use $\ov{A}$ to denote the entry-wise complex conjugate of $A$. For square matrix $A\in \C^{n\times n}$, we use $A^*$ to denote the conjugate transpose of matrix $A$. It is obvious that $A^* = \ov{A}^\top = \ov{A^\top}$. 
We say a complex square matrix $A$ is { Hermitian}, if $A = A^*$, unitary if $AA^*=A^*A=I$, and { positive-semidefinite} (psd) if $A=A^*$ and
$x^* A x \geq 0$ for all $x\in \C^n$. We use $\succeq, \preceq$ to denote the semidefinite ordering, e.g. $A \succeq 0$ means that $A$ is psd.

For $p\in [1,\infty)$, define the {Schatten} $p$-norm of $A$ as
\begin{align*}
\| A \|_p = \left( \sum_{n\geq 1} s_n^p(A) \right)^{1/p}
\end{align*}
for $s_1(A) \geq s_2(A) \geq \cdots \geq s_n(A) \geq \cdots \geq 0$ the singular values of $A$, i.e., the eigenvalues of the Hermitian matrix $|A| = \sqrt{ (A^* A) }$. Then $\| A \|_p^p = \tr [ |T|^p ]$.

For matrix $A \in \C^{n\times n}$, we define $\| A \|$ to be the spectral norm of $A$, i.e.,
\begin{align*}
\|A\| = \max_{ \|  x \|_2 =1, x \in \C^n } x^* A x
\end{align*}

\noindent {\bf Tensor Products.} Given two vectors ${v} \in \C^{d_1}$ and ${w} \in \C^{d_2}$, their
tensor product ${v} \otimes {w} \in \C^{d_1 d_2}$ is the vector
whose $(i,j)^{\text{th}}$ entry is ${v}(i) {w}(j)$ (for concreteness, assume the
entries are in lexicographic order). 
Given two matrices $A_1 \in \C^{d_1 \times d_1}$ and $A_2 \in \C^{d_2 \times d_2}$, their tensor
product $A_1 \otimes A_2 \in \C^{d_1 d_2 \times d_1 d_2}$ is the matrix whose $((i,k),
(j,l))^{\text{th}}$ entry is $A_1(i,j) A_2(k,l)$. It is easy to see that
$$(A\otimes B)({v}\otimes{w}) = A{v}\otimes B{w}$$
and
$$(A\otimes B)(C\otimes D)=AB\otimes CD.$$

For a matrix $X \in \C^{d\times d}$, $\vect(X) \in \C^{d^2}$ will denote the vectorized version of the matrix $X$. That is 
$$
\vect(X) = \sum_{i,j=1}^d X({i,j}) e_i \otimes e_j
$$
We have the following relationship between matrix multiplication and the tensor
product:
$$\vect(AXB) = (A\otimes B^\top )\vect(X).$$

\noindent {\bf Exponential and Logarithm.} All logarithms will be taken with the base $e$, and $\expon(x)$ will denote $e^x$.
The matrix exponential of a complex matrix $A \in \C^{d \times d}$ is defined by the
Taylor expansion: $$ \expon(A) = \sum_{j=0}^{\infty} \frac{A^j}{j!}, $$
which converges for all matrices $A$. 
We will use the fact that
$$\expon(A)\otimes\expon(B) = \expon(A\otimes I + I\otimes B),$$
which may be checked by expanding both sides and comparing terms.

The matrix logarithm of a positive definite matrix $A=UDU^*$ with $D$ diagonal and positive is defined by
$$\log(A) :=U\log(D)U^*,$$
where the logarithm of $D$ is taken entrywise. For such matrices we have
$$\log(\exp(A))=\exp(\log(A))=A.$$
For positive definite $A$ and complex $z$, we define $A^z:=\exp(z\log(A))$.

\noindent {\bf Polar Decomposition.} The polar decomposition of a square complex matrix $A$ is a matrix decomposition of the form 
\begin{align*}
A = U V
\end{align*}
where $U$ is a unitary matrix and $V$ is a psd matrix.
The polar decomposition separates matrix $A$ into a component that stretches the space along a set of orthogonal axes, represented by $V$, and a rotation (with possible reflection) represented by $U$. The decomposition of the complex conjugate of matrix $A$ can written as 
\begin{align*}
\ov{A} = \ov{U} \ov{V}
\end{align*} 
The decomposition of the conjugate transpose of matrix $A$ can be written as 
\begin{align*}
A^* = V^* U^*.
\end{align*}

The following simple proposition will be useful in our proofs. 
\begin{proposition}\label{prop:trace_norm} Let $A$ and $B$ be Hermitian psd matrices. Then
$$
\tr[AB] \le \|A\| \cdot \tr[B]
$$
\end{proposition}

\begin{proof} Let $B = \sum_{j=1}^n \sigma_j v_j v_j^{\dagger}$ be the eigenvalue decomposition of $B$. Then
$$
\tr[AB] =  \sum_{j=1}^n \sigma_j \tr\left[ A v_j v_j^{\dagger} \right] = \sum_{j=1}^n \sigma_j  v_j^{\dagger}  A v_j \le  \sum_{j=1}^n \sigma_j \cdot \|A\| = \|A\| \cdot \tr[B]
$$
\end{proof}

\subsection{Complex analysis}
A function $f:U\rightarrow\C$ on a domain $U\subseteq \C$ is {\em
holomorphic} if it has a complex derivative in a neighborhood of every point
$z\in U$. The existence of a complex derivative in a neighborhood is a very
strong condition, for it implies that any holomorphic function is actually
infinitely differentiable and equal to its own Taylor series (i.e., {\em
analytic}) at every point in $U$. A {\em biholomorphic} function is a
bijective holomorphic function whose inverse is also holomorphic. It follows
from the definition that sums, products, and compositions of holomorphic functions
are holomorphic. We will also talk about matrix-valued holomorphic functions $f:U\rightarrow\C^{d\times d}$, which 
just means that every entry is holomorphic.

The main property that we will use is that the value of a holomorphic function
at a point $z\in U$ can be related to values that it takes on the boundary of
$U$, in the following way. A function $f:U\rightarrow\R\cup\{-\infty\}$ is called {\em subharmonic} if it is 
upper semicontinuous and 
$$ f(z)\le \frac{1}{2\pi}\int_0^{2\pi}f(z+re^{i\theta})d\theta$$
for all $z\in U$ and $r>0$ such that the closed disk $D(z,r)$ is contained in $U$, and all of the
above integrals converge. 

We will make frequent use of the following standard fact.
\begin{proposition} If $f$ is analytic on a domain $U\subset \C$ then $\log|f(z)|$ is subharmonic on $U$.
\end{proposition}

Our main tool will be the Poisson Integral Formula for subharmonic functions.
\begin{lemma}[Poisson integral formula on unit disk {\cite[Eq 1.3.35]{g08}}] \label{lem:poisson_int} For any subharmonic function $U$ defined on the unit disk $\{ z \in \mathbb{C} : |z| \leq 1 \}$, we have that 
\begin{align*}
U(z) \leq \frac{1}{2\pi} \int_{-\pi}^{\pi} U(e^{\i \varphi}) \frac{1-|z|^2}{ | e^{\i \varphi} - z|^2 } \mathrm{d} \varphi, ~\quad~ \forall ~ | z | < 1. 
\end{align*}
\end{lemma}

\section{New Golden-Thompson inequality}\label{sec:gt}

We begin by giving an outline of the proof of Theorem \ref{thm:our_gt}. The
proof of Theorem \ref{thm:generalGT} in \cite{Sutter2017} relies on the
multivariate Lie-Trotter product formula (e.g. see \cite{b97}), which states
that: 
\[
\exp \left( \sum_{j=1}^{k} L_{j} \right)=\lim_{\theta\rightarrow0^+}\left(\prod_{j=1}^{k}\exp(\theta L_{j})\right)^{\frac{1}{\theta}}.
\]
For Hermitian $L_j$, a judicious application of the above allows one to rewrite the trace of the exponential as a limit of
Schatten norms:
\begin{align*}
 \log \left( \tr\left[ \exp \left( \sum_{j=1}^k L_j \right) \right] \right) 
	=  & ~ \lim_{\theta \rightarrow 0^+} \frac{2}{\theta}  \log \left\|   \prod_{j=1}^k \exp \left( \frac{\theta}{2} L_j \right)\right\|_{2/\theta}.
\end{align*}
Thus, understanding the matrix exponential of a sum is the same as
understanding the behavior of a certain norm of the product as $\theta\rightarrow 0$.  The idea
of \cite{Sutter2017} is to use {\em complex interpolation}, along the lines of the Stein-Hirschman theorem
in complex analysis:  for every fixed
real $\theta$ near zero, find a complex function $F_\theta(z)$ that agrees with the
right hand side at $z=\theta$ and is holomorphic on the strip $\{0\le
\Re(z)\le 1\}$. Since the value of a holomorphic function at any point can be
related to an integral of its values on the boundary, this allows one to relate
$F_\theta(\theta)$ to its integrals on $\{\Re(z)=0\}$ and $\{\Re(z)=1\}$, which
are easy to understand. Taking the limit in $\theta$ yields Theorem
\ref{thm:generalGT}.

To avoid integration on the whole vertical line $\left\{ 1+\i b :
b\in\mathbb{R}\right\}$, we observe that the above strategy
only relies on the fact that $\theta$ is enclosed by the two vertical lines
$\left\{ \i b : b\in\mathbb{R}\right\} $ and $\left\{ 1 + \i b :
b\in\mathbb{R}\right\} $, and we could have used any other region enclosing a neighborhood
of real positive $\theta$ near zero, provided we can define the required holomorphic functions $F_\theta$.
We choose the half-circle (which is easy to work with because the Riemann map to the unit disk is
explicit) and use it to derive a variant of the Riesz\textendash Thorin theorem
(Theorem \ref{thm:better_integral}), from which our new multi-matrix Golden-Thompson inequality follows by mimicking the remainder of the proof of
\ref{thm:generalGT} given in \cite{Sutter2017}.
\subsection{Complex estimate on the half disk}

In general, we can upper bound the value of any subharmonic function on a simply connected domain by mapping the domain to the unit disk via Riemann mapping theorem and applying the Poisson integral formula. In this section, we will give such estimate on a unit half disk.
\begin{figure}[!t]
  \centering
    \includegraphics[width=1.0\textwidth]{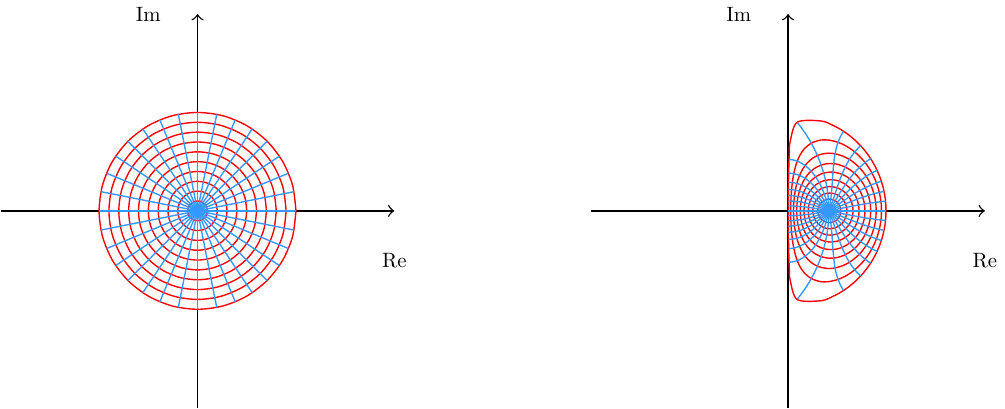}
    \caption{\label{fig:unit_to_half} The function $h(z) = - \frac{1+z}{1-z} + \sqrt{ \left(\frac{1+z}{1-z}\right)^2 + 1}  $ maps the unit disk $\{ z \in \C : |z| \leq 1 \}$ to the half disk $\{ z \in \C : |z|\leq 1 \text{~and~} \text{Re}(z) \geq 0 \}$.}
\end{figure}

The following lemma follows from the biholomorphic map from unit disk onto the half disk. defined in Figure \ref{fig:unit_to_half}.

\begin{lemma}\label{lem:gt_basic}[Poisson Integral Formula on the Half-Disk]
For any analytic function $F$ on the half disk $\{ z  \in \C  : |z| \leq 1 \text{~and~} \mathrm{Re}(z) \geq 0 \}$, we have that
\begin{align}
\log |F(x)| \leq \frac{1}{2\pi} \int_{-\pi}^{\pi} \log \big|F \circ h(e^{\i \varphi}) \big| \frac{1-\rho^2}{1-2\rho \cos ( \varphi) + \rho^2} \mathrm{d} \varphi \label{eq:poisson_half_disk}
\end{align}
for any $0 \leq x \leq 1$ where 
\begin{align*}
h(z) = - \frac{1+z}{1-z} + \sqrt{ \left(\frac{1+z}{1-z}\right)^2 + 1} \quad \mathrm{~and~} \quad \rho = \frac{x^2+2x-1}{x^2-2x -1}.
\end{align*}
\end{lemma}

\begin{proof}
Note that the function $h(z)$ is a biholomorphic map from the unit disk $\{ z \in \C : |z| \leq 1 \}$ to the half disk $\{ z \in \C : |z|\leq 1 \text{~and~} \text{Re}(z) \geq 0 \} $. (We provide a proof for completeness, see Lemma~\ref{lem:h_map} \footnote{The similar version is an exercise 4 in page 163 (Section VII) of \cite{c78}, and also can be found here, \url{https://math.stackexchange.com/questions/882147/find-a-conformal-map-from-semi-disk-onto-unit-disk}})

Since $F$ and $h$ are holomorphic, $\log |F \circ h (z)|$ is subharmonic. Therefore, we can apply the Poisson integral formula for subharmonic functions (Lemma \ref{lem:poisson_int}) and get
\begin{align*}
\log |F \circ h(z)| \leq & ~ \frac{1}{2\pi } \int_{-\pi}^{\pi} \log \big|F \circ h (e^{\i \varphi})\big| \frac{1-|z|^2}{ | e^{\i \varphi} - z |^2 } \mathrm{d} \varphi \\ 
= & ~ \frac{1}{2\pi} \int_{-\pi}^{\pi} \log \big|F \circ h (e^{\i \varphi})\big| \frac{ 1- \rho^2}{ 1 - 2\rho \cos (\theta-\varphi) + \rho^2}  \mathrm{d} \varphi
\end{align*}
for $z = \rho e^{\i \theta}$.

Setting $x=h(z)$, we obtain that
\begin{align*}
z = \frac{x^2 +2 x - 1}{ x^2 - 2 x -1}.
\end{align*}
Therefore, we have that
\begin{align*}
\theta = 0 \text{~and~} \rho = \frac{x^2+2x -1}{ x^2 -2x - 1}.
\end{align*}
This gives the desired result.

\end{proof}

The following lemma allows us to conveniently study the behavior of certain
analytic functions near zero. The idea is that when $|F(z)|$ is at most $1$ on
the imaginary axis, $\log | F(\theta) |$ should be close to $0$ for small $\theta$,
 and the value of $\log | F(\theta) |/\theta$ can be upper
bounded by a suitable average of the values on the boundary of the half disk.

\begin{lemma} \label{lem:half_disk_estimate}
Given any analytic function $F$ on the half disk $\{ z  \in \C  : |z| \leq 1 \text{~and~} \mathrm{Re}(z) \geq 0 \}$. Suppose that $|F(\i y)| \leq 1$ for all $y \in [-1,1]$. Then, for any $0\leq \theta \leq 1/4$, we have
\begin{align*}
\log | F(\theta) | \leq \left(\frac{4\theta}{\pi} + O(\theta^2) \right) \int_{-\pi/2}^{\pi/2}  \log \big|F (e^{\i \phi}) \big|  \mathrm{d} \mu_\theta (\phi)
\end{align*}
where $\mu_\theta$ is some probability distribution on $[-\frac{\pi}{2},\frac{\pi}{2}]$ depending only on $\theta$, and $\mu_\theta \rightarrow$ some probability distribution as $\theta \rightarrow 0^+$.
\end{lemma}


\begin{proof}
Since $\log |F(\i y)| \leq 0$ for all $y \in [-1,1]$ and since $h(e^{\i \varphi})$ is imaginary with modulus  at most $1$ whenever $|\varphi| \leq \pi /2 $, we can ignore these $\varphi$ in the integral (\ref{eq:poisson_half_disk}), namely,
\begin{align}
\log |F(x)| & \leq \frac{1}{2\pi} \int_{-\pi}^{\pi} \log \big|F \circ h(e^{\i \varphi}) \big| \frac{1-\rho^2}{1-2\rho \cos ( \varphi) + \rho^2} \mathrm{d} \varphi \nonumber \\
& \leq \frac{1}{2\pi} \int_{\pi/2 \leq|\varphi| \leq \pi} \log \big|F \circ h(e^{\i \varphi}) \big| \frac{1-\rho^2}{1-2\rho \cos ( \varphi) + \rho^2} \mathrm{d} \varphi. \label{eq:poisson_half_disk2}
\end{align}

To bound the right hand side, for $\pi/2 \leq |\varphi| \leq \pi$ and $0 \leq \rho \leq 1$, we can prove the following statement using elementary calculations (see Lemma~\ref{lem:theta_square}),
\begin{align*}
\frac{1-\rho^2}{1-2\rho \cos ( \varphi) + \rho^2} 
 = \frac{1-\rho}{1 - \cos ( \varphi)} \pm O(1-\rho)^2.
\end{align*}
Note that $\rho = \frac{\theta^2 + 2\theta-1}{\theta^2-2\theta-1} \geq 1-4\theta$ for all $0 \leq \theta \leq 1/4$. Therefore, we have that $0 \leq \rho \leq 1$ and 
$$ \frac{1-\rho^2}{1-2\rho \cos ( \varphi) + \rho^2} \le \frac{4 \theta}{1 - \cos ( \varphi)} + O(\theta^2)$$

Putting this inequality into (\ref{eq:poisson_half_disk2}), we have that
\begin{align*}
\log |F(x)| & \leq \frac{1}{2 \pi} \int_{\pi/2 \leq|\varphi| \leq \pi} \left( \frac{4 \theta}{1 - \cos ( \varphi)} + O(\theta^2) \right) \log |F \circ h(e^{\i \varphi}) | \mathrm{d} \varphi  .
\end{align*}

We now observe that
\begin{align*}
\int_{\pi/2 \leq|\varphi| \leq \pi} \frac{1}{1 - \cos ( \varphi)} \mathrm{d} \varphi = \int_{-\pi}^{-\pi/2} \frac{1}{1 - \cos ( \varphi)} \mathrm{d} \varphi + \int_{\pi/2}^{\pi} \frac{1}{1 - \cos ( \varphi)} \mathrm{d} \varphi  =  2.
\end{align*}
Note that $h$ maps $e^{\i \varphi}$ for $\pi/2 \le |\varphi| \le \pi$ to the boundary of the half disk ($[-\pi/2,\pi/2]$), and let $\mu_\theta$ be some probability distribution on $[-\frac{\pi}{2},\frac{\pi}{2}]$ depending only on $\theta$. Then we can have
\begin{align*}
\log | F(\theta) | \leq \left(\frac{4\theta}{\pi} + O(\theta^2) \right) \int_{-\pi/2}^{\pi/2}  \log \big|F (e^{\i \phi}) \big|  \mathrm{d} \mu_\theta (\phi).
\end{align*}
Note that $\mu_{\theta} \rightarrow$ some probability distribution as $\theta \rightarrow 0^+$.
\end{proof}

\subsection{Bounded Multimatrix Golden-Thompson type inequality}
Plugging our new complex estimate into the proof of Theorem 3.1 in \cite{Sutter2017}, we obtain the following Riesz-Thorin-type inequality. We give a complete proof for completeness.
\begin{theorem}[Riesz-Thorin-type inequality]\label{thm:better_integral}
Let $S=\{ z \in \C : |z|\le 1 \mathrm{~and~} \mathrm{Re}(z) \geq 0 \}$ and let $G$ be a holomorphic map from $S$ to square matrices. Let $p_0 \geq p_1\in [1,\infty]$, for $\theta \in (0,1)$, define $p_{\theta}$ by
\begin{align*}
\frac{1}{p_{\theta}} = \frac{1-\theta}{p_0} + \frac{\theta}{p_1}.
\end{align*} 
If $z \rightarrow \| G(z) \|_{p_{\mathrm{Re}(z)}}$ is uniformly bounded on $S$ and $\| G(\i t) \|_{p_0} \leq 1$, then for any $0 \leq \theta \leq 1/4$,
\begin{align*}
\log \| G(\theta) \|_{p_\theta} \leq \left(\frac{4\theta}{\pi} + O(\theta^2) \right) \int_{-\pi/2}^{\pi/2} \log \| G(e^{\i \phi}) \|_{p_1} \mathrm{d} \mu_\theta (\phi)
\end{align*}
where $\mu_\theta$ is some probability distribution on $[-\frac{\pi}{2},\frac{\pi}{2}]$ depending only on $\theta$.
\end{theorem}

\begin{proof}
To apply Lemma \ref{lem:half_disk_estimate}, we want to define a holomorphic function $F(z)$ such that 
$$|F(\i t)| \leq 1 ~ , ~ |F(z)| \leq \| G(z) \|_{p_1} \text{~and~} F(\theta ) = \| G(\theta) \|_{p_{\theta}}.$$

Now, we describe how to define such $F(z)$. For $x \in [0,1]$, define $q_x$ as the H\"{o}lder conjugate of $p_x$ such that $p_x^{-1} + q_x^{-1} = 1$. Hence, using the definition of $p_x$ in the statement, we have
\begin{align*}
\frac{1}{q_x}= \frac{1-x}{q_0} + \frac{x}{q_1}.
\end{align*}
	Now for our fixed $\theta \in (0,1)$, let $G(\theta) = U V$ be the polar decomposition of $G(\theta)$, where $V$ is positive definite since $G(\theta)$ is always invertible, and $U$ is unitary. Finally, we define $X(z)$ and $F(z)$ by
\begin{align*}
X(z)^* & = (V/c)^{ p_\theta ( \frac{1-z}{q_0} + \frac{z}{q_1} ) } U^*, \text{~where~} c = \| V \|_{p_\theta} = \| G(\theta) \|_{p_{\theta}} \\
F(z) & = \tr[  X(z)^* G(z) ].
\end{align*}

Note that $F(z)$ is holomorphic. Due to the renormalization $c$, we can show $\| X (x+\i y)) \|_{q_x}^{q_x} = 1$ for all $x \in [0,1]$:
\begin{align*}
\| X(x+\i y) \|_{q_x}^{q_x} = & ~ \tr \left[ \sqrt{ X^*(x+\i y) X(x+\i y) }^{q_x} \right] & \text{~by~definition~of~}\| \cdot \|_p \\
= & ~ \tr\left[ (V/c)^{q_x p_\theta ( \frac{1-x}{q_0} + \frac{x}{q_1} )} \right]  & \text{~by~} U^* U = I \\
= & ~ \tr\left[ (V/c)^{  p_\theta} \right]  & \text{~by~} \frac{1}{q_x}= \frac{1-x}{q_0} + \frac{x}{q_1} \\
= & ~ 1. & \text{~by~} c = \| V \|_{p_\theta}
\end{align*}
Therefore, $F(z)$ is bounded on $S$ as follows
\begin{align*}
|F(x+ \i y)| \leq \| X(x+\i y) \|_{q_x} \cdot \| G(x+\i y) \|_{p_x} \leq \| G(x+ \i y) \|_{p_x}
\end{align*}
Using this, it is obvious that 
\begin{align*}
|F(\i t)| \leq \| G(\i t) \|_{p_0} \leq 1, \text{~and~} |F(z)| \leq \| G(z) \|_{p_{\mathrm{Re}(z)}} \leq \| G(z) \|_{p_1}.
\end{align*}
for all $z \in S$ where we used that $p_0 \geq p_{\mathrm{Re}(z)} \geq p_1$.

Finally, we verify that $F(\theta ) = \| G(\theta) \|_{p_{\theta}}$:
\begin{align*}
F(\theta) = & ~ \tr[ X(\theta)^* G(\theta) ] \\
 = & ~ \tr[ (V/c)^{p_{\theta} ( \frac{1-\theta}{q_0} + \frac{\theta}{q_1} ) } U^*  \cdot U V ] & \text{~by~definition~of~}X \text{~and~} G\\
 = & ~ \tr[ (V/c)^{p_{\theta} ( \frac{1-\theta}{q_0} + \frac{\theta}{q_1} ) }  V ] & \text{~by~} U^* U = I \\
 = & ~ \tr[ c^{-p_{\theta}/q_{\theta}} V^{1+ p_{\theta}/q_{\theta}} ] & \text{~by~} \frac{1}{q_\theta}= \frac{1-\theta}{q_0} + \frac{\theta}{q_1}  \\
 = & ~ \tr[ c^{1-p_{\theta}} V^{p_{\theta}} ] & \text{~by~} p_{\theta}/q_{\theta} = p_{\theta} (1-1/p_{\theta}) = p_{\theta } - 1 \\
 = & ~ c^{1-p_{\theta} } c^{p_{\theta}} & \text{~by~} (\tr[V^{p_{\theta}}])^{1/p_{\theta}} = c \\
 = & ~ \| G(\theta) \|_{p_\theta} & ~ \text{~by~} c= \| G(\theta) \|_{p_{\theta}}.
\end{align*}

Hence, the statement follows from Lemma \ref{lem:half_disk_estimate}.
\end{proof}

Now, we are ready to prove our variant of multimatrix Golden-Thompson inequality. It follows from plugging in Theorem \ref{thm:better_integral} into the proof of Theorem 3.5 in \cite{Sutter2017}. We give a complete proof for completeness.

\begin{theorem}[Multimatrix Golden-Thompson inequality]\label{thm:better_gt}
For any $k$ Hermitian matrices $H_1, \cdots, H_k$, we have:
\begin{align*}
\log \left( \tr \left[ \exp \left( \sum_{j=1}^k H_j \right) \right] \right) \leq \frac{4}{\pi} \int_{-\frac{\pi}{2}}^{\frac{\pi}{2}} \log \left( \tr \left[ \prod_{j=1}^k \exp \left( \frac{e^{\i \phi}}{2} H_j \right) \prod_{j=k}^1 \exp \left( \frac{ e^{-\i \phi} }{2} H_j \right) \right] \right) \mathrm{d} \mu (\phi)
\end{align*}
where $\mu$ is some probability distribution on $[-\frac{\pi}{2},\frac{\pi}{2}]$.
\end{theorem}


\begin{proof}
Define
$$G(z) = \prod_{j=1}^k \exp \left( \frac{z}{2} H_j \right).$$
Note that $\| G(\i y)  \|_\infty = 1$ for all $y \in \R$. We now have:
\begin{align*}
	& ~ \log \left( \tr\left[ \exp \left( \sum_{j=1}^k H_j \right) \right] \right) \\
	= & ~ \log \left( \tr\left[ \exp \left( \sum_{j=1}^k H_j/2+\sum_{j=k}^1 H_j/2 \right) \right] \right) \\
	\\&\quad\textrm{by the Lie-Trotter formula}\\
	= & ~ \log \left( \tr\left[\lim_{\theta\rightarrow 0^+} (G(\theta)G(\theta)^*)^{1/\theta}\right]\right)\\
	&\quad\textrm{since the $H_j$ are Hermitian}\\
	= & ~ \lim_{\theta\rightarrow 0^+}\frac{2}{\theta}\log \left( \tr\left[(G(\theta)G(\theta)^*)^{1/\theta}\right]^{\theta/2}\right)\\
	\\&\quad\textrm{by continuity of $\log$ away from $0$}\\
=  & ~ \lim_{\theta \rightarrow 0^+} \frac{2}{\theta} \log \left\| G(\theta)  \right\|_{2/\theta} \\
\leq & ~ \lim_{\theta \rightarrow 0^+} 2 \left(\frac{4}{\pi} + O(\theta) \right) \int_{-\pi/2}^{\pi/2}  \log \| G(e^{\i \phi}) \|_2  \mathrm{d} \mu_\theta (\phi) \\
	\\&\quad\textrm{ by Theorem~\ref{thm:better_integral} with $p_0 = \infty$ and $p_1 = 2$}\\
= & ~ \lim_{\theta \rightarrow 0^+}  \left(\frac{4}{\pi} + O(\theta) \right) 
\int_{-\pi/2}^{\pi/2}  \log \left( \tr \left[ G(e^{\i \phi}) G(e^{\i \phi} )^* \right] \right) \mathrm{d} \mu_\theta (\phi),
\end{align*}

When $\theta \rightarrow 0^+$, $\mu_{\theta}(\phi)$ converges to some probability distribution $\mu$. This completes the proof.
\end{proof}

\begin{remark} 

We suspect the constant $4/\pi$ is tight and that any constant larger than one is unavoidable if we consider the maximum over a bounded domain inside the strip $\{ z \in \C : 0 \leq \text{Re}(z) \leq 1 \}$.
\end{remark}

\section{Proof of Theorem \ref{thm:main}}\label{sec:mainproof}

In this section we present the proof of Theorem \ref{thm:main}. We restate it here (with explicit constants) for convenience. 

\begin{theorem}\label{thm:main_restate} Let $G$ be a regular $\lambda$-expander on $V$ and let $f$ be a function $f : V \rightarrow \C^{d\times d}$ s.t. 
\begin{enumerate}
\item For each $v \in V$, $f(v)$ is Hermitian and $\|f(v)\| \le 1$. 
\item $\sum_{v \in V} f(v) = 0$. 
\end{enumerate}
If $v_1,\ldots,v_k$ is a stationary random walk on $G$, and $\epsilon\in (0,1)$,
\begin{align*}
&\Pr \left[ \lambda_{\textnormal{max}} \left( \sum_{j=1}^k f(v_j)\right) \ge k \eps\right] \le d^{2-\pi/4} \cdot \expon\left(-  \eps^2 (1 - \lambda) k/ 80 \right) \\
&\Pr \left[ \lambda_{\textnormal{min}} \left( \sum_{i=1}^k f(v_j)\right) \le -k \eps\right] \le d^{2-\pi/4} \cdot \expon\left(-  \eps^2 (1 - \lambda) k /80 \right)
\end{align*} 
\end{theorem}
\begin{remark}
Depsite the exponent of $d$ is different from Theorem \ref{thm:main}, we note that since the left hand side (the probability) is at most $1$, one can prove the same statement with any positive exponent by changing the constant $80$.
\end{remark}
\begin{proof}Due to symmetry, it suffices to prove just one of the statements. Let $t > 0$ be a parameter to be chosen later. Then
\begin{align}
\Pr \left[ \lambda_{\textnormal{max}} \left( \sum_{j=1}^k f(v_j)\right) \ge k \eps\right] &\le  \Pr \left[ \tr\left[ \expon\left(t \sum_{j=1}^k f(v_j) \right)\right] \ge \expon(tk\eps)\right] \nonumber \\
&\le \frac{\E \left[ \tr\left[ \expon\left(t \sum_{j=1}^k f(v_j) \right)\right] \right]}{\expon(tk\eps)} \label{eqn:ankit1}
\end{align}
The second inequality follows from Markov's inequality.

Now the question is how to bound $\E_{v_1,\cdots,v_k} [ \tr [ \exp(t\sum_{j=1}^k f(v_j)) ] ]$. Using Theorem~\ref{thm:better_gt} and note that $\mu(\phi)$ is a probability distribution on $[-\frac{\pi}{2}, \frac{\pi}{2}]$, we have
\begin{align*}
 & ~  \log \left( \tr \left[ \exp \left( t \sum_{j=1}^k f(v_j) \right) \right] \right) \\
\leq & ~ \frac{4}{\pi}  \int_{-\frac{\pi}{2}}^{\frac{\pi}{2}}  \log \tr \left[ \prod_{j=1}^k \exp \left( \frac{e^{\i \phi}}{2} t f( v_j) \right) \prod_{j=k}^1 \exp \left( \frac{ e^{-\i \phi} }{2} t f( v_j) \right) \right]  \mathrm{d} \mu (\phi) \\
\leq & ~ \frac{4}{\pi} \log \int_{-\frac{\pi}{2}}^{ \frac{\pi}{2} } \tr  \left[ \prod_{j=1}^k \exp \left( \frac{e^{\i \phi}}{2} t f( v_j) \right) \prod_{j=k}^1 \exp \left( \frac{ e^{-\i \phi} }{2} t f( v_j) \right) \right] \mathrm{d} \mu (\phi)
\end{align*}
where the the second step follows by concavity of $\log$ function. This implies that
\begin{align}\label{eq:better_gt_without_log_part1}
 \tr \left[ \exp \left( t \sum_{j=1}^k f(v_j) \right) \right] \leq \left(  \int_{-\frac{\pi}{2}}^{\frac{\pi}{2}} \tr  \left[ \prod_{j=1}^k \exp \left( \frac{e^{\i \phi}}{2} t f( v_j) \right) \prod_{j=k}^1 \exp \left( \frac{ e^{-\i \phi} }{2} t f( v_j) \right) \right]  \mathrm{d} \mu (\phi) \right)^{ \frac{4}{\pi} }
\end{align}
Note that $\|x\|_p \leq d^{1/p-1} \| x \|_1$ for $p\in (0,1)$, choosing $p=\pi/4$ we have
\begin{align}\label{eq:better_gt_without_log_part2}
\left( \tr \left[ \exp \left(\frac{\pi}{4} t \sum_{j=1}^k f(v_j) \right) \right] \right)^{\frac{4}{\pi}} \leq  d^{4/\pi - 1} \tr \left[ \exp \left(t \sum_{j=1}^k f(v_j) \right) \right].
\end{align}
Combining Eq.~\eqref{eq:better_gt_without_log_part1} and Eq.~\eqref{eq:better_gt_without_log_part2}, we have 
\begin{align}\label{eq:better_gt_without_log}
 \tr \left[ \exp \left( \frac{\pi}{4} t \sum_{j=1}^k f(v_j) \right) \right] \leq d^{1 - \pi/4}  \int_{-\frac{\pi}{2}}^{\frac{\pi}{2}} \tr  \left[ \prod_{j=1}^k \exp \left( \frac{e^{\i \phi}}{2} t f( v_j) \right) \prod_{j=k}^1 \exp \left( \frac{ e^{-\i \phi} }{2} t f( v_j) \right) \right]  \mathrm{d} \mu (\phi) 
\end{align}

The core of the proof is the following bound on the moment generating
function-like expression that appears above, 
thinking of $\i \phi$ as $\gamma + \i b$ with $\gamma^2 + b^2 = 1$:
\begin{lemma}\label{lemma:smallb} 
Let $G$ be a regular $\lambda$-expander on $V$, let $f$ be a function $f: V \rightarrow \mathbb{C}^{d\times d}$ and $\sum_{v \in V} f(v) = 0$, let $v_1, \cdots, v_k$ be a stationary random walk on $G$, for any $t>0$, $\gamma \geq 0$, $b>0$, $t^2(\gamma^2 +b^2) \leq 1$, and $t\gamma \leq \frac{1-\lambda}{4\lambda}$ we have
\begin{align*}
\E \left[ \tr\left[ \prod_{j=1}^k \exp \left( \frac{t f(v_j) ( \gamma + \i b ) }{2} \right) \prod_{j=k}^1 \exp \left( \frac{ t f(v_j) ( \gamma - \i b ) }{2} \right) \right] \right]  \leq d \cdot \exp \left( kt^2 (\gamma^2 + b^2) ( 1 +  \frac{ 8 }{ 1-\lambda } )  \right).
\end{align*}
\end{lemma}

Assuming this lemma, we can easily complete the proof of the theorem as:
\begin{align}\label{eq:bounding_main_expectation}
 & ~ \E_{v_1,\cdots,v_k} \left[  \tr \left[ \exp \left( \frac{\pi}{4} t \sum_{j=1}^k f(v_j) \right) \right] \right] \notag \\
\leq & ~ d^{1-\pi/4} \E_{v_1,\cdots,v_k} \left[  \int_{-\frac{\pi}{2} }^{\frac{\pi}{2}} \tr  \left[ \prod_{j=1}^k \exp \left( \frac{e^{\i \phi}}{2} t f( v_j) \right) \prod_{j=k}^1 \exp \left( \frac{ e^{-\i \phi} }{2} t f( v_j) \right) \right]   \mathrm{d} \mu (\phi) \right] \notag  \\
= & ~ d^{1-\pi/4} \int_{-\frac{\pi}{2}}^{\frac{\pi}{2}} \E_{v_1,\cdots,v_k}  \left[ \tr  \left[ \prod_{j=1}^k \exp \left( \frac{e^{\i \phi}}{2} t f( v_j) \right) \prod_{j=k}^1 \exp \left( \frac{ e^{-\i \phi} }{2} t f( v_j) \right) \right] \right] \mathrm{d} \mu (\phi) \notag  \\
\leq & ~ d^{1-\pi/4}  \int_{-\frac{\pi}{2}}^{\frac{\pi}{2}} d \exp \left( kt^2 | e^{\i \phi} |^2 (1+ \frac{8}{1-\lambda}) \right)   \mathrm{d} \mu (\phi) \notag  \\
= & ~ d^{2-\pi/4} \exp \left(kt^2 (1+ \frac{8}{1-\lambda}) \right) \int_{-\frac{\pi}{2}}^{\frac{\pi}{2}}  \mathrm{d} \mu ({\phi}) \notag \\
= & ~ d^{2-\pi/4} \exp \left(kt^2 (1+ \frac{8}{1-\lambda}) \right) ,
\end{align}
where the first step follows by the Equation~\eqref{eq:better_gt_without_log}, the second step follows by swapping $\E$ and $\int$, the third step follows by Lemma~\ref{lemma:smallb} , the fourth step follows by $|e^{\i \phi}|=1$, and the last step follows by $\int_{-\frac{\pi}{2}}^{\frac{\pi}{2}} \mathrm{d} \mu (\phi) = 1$. 

Finally, putting it all together, 
\begin{align*}
\Pr_{v_1,\cdots,v_k} \left[\lambda_{\max} \left( \sum_{j=1}^k f(v_j) \right) \geq k\epsilon \right] \leq & ~ d^{2-\pi/4} \cdot \exp \left( (4/\pi)^2 k t^2 \frac{9}{1-\lambda} - kt \epsilon \right) \\ 
= & ~ d^{2-\pi/4} \cdot \exp \left( (4/\pi)^2 k \epsilon^2 (1-\lambda)^2 \frac{1}{36^2} \frac{9}{1-\lambda} - k \frac{(1-\lambda) \epsilon}{36} \epsilon \right) \\
\leq & ~ d^{2-\pi/4} \cdot \exp \left( - k \epsilon^2 (1-\lambda) /72 \right).
\end{align*}
where the first step follows by Eq.~\eqref{eq:bounding_main_expectation} the second step follows by choosing $t=(1-\lambda)\epsilon /36$.
\end{proof}

We now give the proof of Lemma \ref{lemma:smallb}
\begin{proof}[Proof of Lemma \ref{lemma:smallb}]
We start by writing the expected trace expression in terms of the transition
matrix of the random walk. This is an analogue of a step which is common to most of the expander chernoff bound proofs in the scalar case. Let $P$ be the normalized adjacency matrix of $G$ and let $\widetilde{P} = P \otimes I_{d^2}$. Let $E$ denote the $nd^2 \times nd^2$ block diagonal matrix where the $v^{\text{th}}$ diagonal block is the matrix 
\begin{align}\label{eq:def_E_M_v}
M_v = \expon\left( \frac{t f(v)(\gamma+\i b)}{2}\right) \otimes \expon\left( \frac{t f(v) (\gamma -\i b)}{2}\right).
\end{align}
Then $\left(E \widetilde{P}\right)^k$ is an $nd^2 \times nd^2$ block matrix whose $(u,v)^{\text{th}}$ ($d^2 \times d^2$) block is given by the matrix
\begin{align}\label{eq:def_E_wtP_k}
\sum_{v_1,\ldots,v_{k-1}} P_{u,v_1} \cdot \left( \prod_{j=1}^{k-2} P_{v_j, v_{j+1}} \right) \cdot  P_{v_{k-1}, v} \cdot M_u \cdot \left( \prod_{j=1}^{k-1} M_{v_j} \right)
\end{align}
Let $z_0 \in \C^{nd^2}$ be the vector $\frac{\mathbf{1}}{\sqrt{n}} \otimes
\vect(I_d)$. Here $\mathbf{1}$ is the all $1$'s vector and $\vect(I_d)$ is the
vector form of the identity matrix. Then, by applying 
$$
\langle \vect(I_d), A_1 \otimes A_2 \: \vect(I_d) \rangle = \tr\left[ A_1 A_2^T\right]
$$
it follows that for a stationary random walk $v_1,\ldots,v_k$:
\begin{align*}
\E \left[ \tr\left[ \prod_{j = 1}^k \expon\left( \frac{t f(v_j)(\gamma +\i b)}{2}\right) \prod_{j = k}^1 \expon\left( \frac{t f(v_j)(\gamma-\i b)}{2}\right)\right] \right]
&=\E \left[ \left\langle \vect(I_d),\prod_{i=1}^k M_{v_i} \vect(I_d)\right\rangle \right]
\\&=
\left\langle z_0, \left(E \widetilde{P}\right)^k z_0 \right\rangle
\end{align*}
Hence we can focus our attention on $\left\langle z_0, \left(E
\widetilde{P}\right)^k z_0 \right\rangle$. 

Let $\mathbf{1}$ be the all $1$'s vector. For a vector $z \in
\mathbb{C}^{nd^2}$, let $z^{\parallel}$ denote the component of $z$ which lies
in the subspace spanned by the $d^2$ vectors $\mathbf{1} \otimes e_i$, $1 \le i
\le d^2$. Let $z^{\bot}$ denote the component in the orthogonal space.
Letting $z_j = \left(E
\widetilde{P}\right)^j z_0$, we are interested in bounding
\begin{align*}
\langle z_0, z_k \rangle = \langle z_0, z_k^{\parallel} \rangle \le \|z_0\| \cdot \|z_k^{\parallel}\| = \sqrt{d} \cdot  \|z_k^{\parallel}\|
\end{align*}

The following lemma is the analogue of the main lemma in Healy's proof for the
scalar valued expander Chernoff bound \cite{healy08}. Roughly speaking, it
tracks how much a vector can move in and out of the subspace we are interested in
as the operator $E\widetilde{P}$ is applied.

\begin{lemma}\label{lemma:healy} Given four parameters $\lambda \in [0,1]$, $\gamma \geq 0$, $\ell \geq 0$, and $t >0$. Let $G$ be a regular $\lambda$-expander on $V$. Suppose each vertex $v$ is assigned a matrix $H_v \in \C^{d^2 \times d^2}$ s.t. $\|H_v\| \le \ell$ and $\sum_v H_v = 0$.  Let $P$ be the normalized adjacency matrix of $G$ and let $\widetilde{P} = P \otimes I_{d^2}$. Let $E$ denote the $nd^2 \times nd^2$ block diagonal matrix where the $v$-th diagonal block is the matrix $\expon(t H_v)$. Also suppose that $\|E\| = \max_{v\in V} \|\expon(t H_v)\| \le \expon(  \gamma t)$.
Then for any $z \in \C^{nd^2}$, we have:
\begin{eqnarray*}
1.& \|(E \widetilde{P} z^{\parallel})^{\parallel}\| \le \alpha_1 \|z^{\parallel}\| & \text{~where~} \alpha_1 = \expon(t \ell) - t \ell \\
2.& \|(E \widetilde{P} z^{\parallel})^{\bot}\| \le  \alpha_2 \|z^{\parallel}\| & \text{~where~} \alpha_2 =  \expon(t \ell) - 1 \\
3.& \|(E \widetilde{P} z^{\bot})^{\parallel}\| \le \alpha_3 \|z^{\bot}\| & \text{~where~} \alpha_3 = \lambda \cdot \left( \expon(t \ell) - 1 \right) \\
4.& \|(E \widetilde{P} z^{\bot})^{\bot}\| \le \alpha_4 \|z^{\bot}\| & \text{~where~} \alpha_4 = \lambda \cdot \exp(t\gamma) 
\end{eqnarray*}
\end{lemma}

\begin{proof}[Proof of Lemma \ref{lemma:healy}]

Part 1.

Note that $(E \widetilde{P} z^{\parallel})^{\parallel} = (E z^{\parallel})^{\parallel}$. Let ${\bf 1} \in \R^n$ denote all ones vector, suppose $z^{\parallel} = \mathbf{1} \tensor \mathbf{w}$ for some $\mathbf{w} \in \C^{d^2}$. Then $\|z^{\parallel}\| = \sqrt{n} \cdot \|\mathbf{w}\|$ and 
\begin{align*}
(E z^{\parallel})^{\parallel} = \mathbf{1} \otimes \left( \frac{1}{n} \sum_{v \in V} \expon(H_v) \mathbf{w} \right).
\end{align*}

We can upper bound $\| \frac{1}{n} \sum_{v \in V} \exp( t H_v)   \|$ in the following way,
\begin{align*}
\left\| \frac{1}{n} \sum_{v \in V} \exp(t H_v)  \right\| = & ~ \left\| \frac{1}{n} \sum_{v\in V} \sum_{j=0}^{\infty} \frac{t^j H_v^j}{j!} \right\| \\
= & ~ \left\| I + \frac{1}{n} \sum_{v\in V} \sum_{j=2}^{\infty} \frac{t^j H_v^j}{j!} \right\| \\
\leq & ~ 1 + \frac{1}{n} \sum_{v\in V} \sum_{j \geq 2} \frac{t^j}{j!} \| H_v \|^j \\
= & ~ 1 + \sum_{j\geq 2} \frac{ (t\ell)^j }{ j! } \\
= & ~ \exp(t\ell) -t\ell,
\end{align*}
where the first step follows by Taylor expansion, the second step follows by $\sum_{v\in V} H_v = 0$, the third step follows by triangle inequality, the fourth step follows by $|V|=n$ and $\| H_v \| \leq \ell$, and last step follows by Taylor expansion.

Thus, 
\begin{align*}
\| (E z^{\parallel} )^{\parallel} \| = \sqrt{n} \left\| \frac{1}{n} \sum_{v\in V} \exp(t H_v) w \right\| \leq \sqrt{n} \| w \| (\exp(t\ell)-t\ell) = \| z^{\parallel} \| (\exp(t\ell)-t\ell).
\end{align*}

Part 2. Note that $(E\wt{P} z^{\parallel})^{\bot} = (E z^{\parallel})^{\bot} = ( (E-I) z^{\parallel} )^{\bot}$. and $(z^{\parallel})^{\bot}=0$. We can upper bound $\| ( (E-I)z^{\parallel} )^{\bot} \|$ in the following way,
\begin{align*}
\| ( (E-I) z^{\parallel} )^{\bot} \| \leq & ~ \| (E-I) z^{\parallel} \| \\
\leq & ~ \| E - I \| \cdot \| z^{\parallel} \| \\
= & ~ \max_{v\in V} \| \exp(t H_v) - I \| \cdot \| z^{\parallel} \| \\
= & ~ \max_{v\in V} \left\| \sum_{j=1}^{\infty} \frac{t^j}{j!} H_v^j \right\| \cdot \| z^{\parallel} \| \\
\leq & ~ \left(\sum_{j=1}^{\infty} \frac{t^j \ell^j}{j!} \right) \cdot \| z^{\parallel} \| \\
= & ~ (\exp(t\ell) - 1) \cdot \| z^{\parallel} \|,
\end{align*}
where the second step follows by $\| A x \| \leq \| A \|\cdot \| x \|$, the third step follows by definition of $E$, the fourth step follows by Taylor expansion, the fifth step follows by triangle inequality and $\| H_v \| \leq \ell$, and the last step follows by Taylor expansion. 

Part 3. Note that 
$$
(E \widetilde{P} z^{\bot})^{\parallel} =((E-I)\widetilde{P} z^{\bot})^{\parallel}
$$
This is because  $(\widetilde{P} z^{\bot})^{\parallel} = 0$ since $\widetilde{P}$ preserves the property of being orthogonal to the space spanned by the vectors $\mathbf{1} \otimes e_i$ (these are the top eigenvectors of $\widetilde{P}$). 
Hence we can bound
\begin{align*}
\left\| ((E-I)\widetilde{P} z^{\bot})^{\parallel} \right\| &\le  \left\| (E-I)\widetilde{P} z^{\bot} \right\| \\
&\le \|E-I\| \cdot  \left\|\widetilde{P} z^{\bot} \right\| \\
&\le \left( \expon(t \ell) - 1 \right)  \cdot \lambda \cdot \left\|z^{\bot} \right\|
\end{align*}
Third inequality follows from the fact that $\|E-I\| \le \expon(t \ell) - 1$ and that $G$ is a $\lambda$-expander. 

Part 4. 
We can bound 
\begin{align*}
\| (E \widetilde{P} z^\bot )^{\bot} \| \leq \| E \widetilde{P} z^{\bot} \| \leq \exp(t\gamma) \cdot \lambda \| z^{\bot} \|
\end{align*}
where the second step follows by $\| E\| \leq \exp(\gamma t)$ and $G$ is a $\lambda$-expander.

\end{proof}

We now use the above Lemma \ref{lemma:healy} to analyze the evolution of $z_j^{\parallel}$ and $z_j^{\bot}$. 
Recall the definition of $H_v$,
\begin{align*}
H_v  = \frac{f(v) (\gamma+\i b)}{2} \otimes I_d + I_d \otimes \frac{f(v) (\gamma-\i b)}{2}.
\end{align*}
which means
\begin{align*}
\exp(t H_v)= & ~ \exp \left( \frac{ t f(v) (\gamma+\i b)}{2} \otimes I_d + I_d \otimes \frac{ t f(v) (\gamma-\i b)}{2} \right)  \\
= & ~ \exp \left( \frac{ t f(v) (\gamma + \i b) }{2} \right) \otimes \exp \left( \frac{t f(v) ( \gamma - \i b)}{2} \right) \\
= & ~ M_v,
\end{align*}
where the the first step follows by definition of $H_v$, the second step follows by $\exp(A\otimes I_d + I_d \otimes B) = \exp (A) \otimes \exp(B)$, and the last step follows by Eq.~\eqref{eq:def_E_M_v}.

We can upper bound $\| H_v \| \leq \sqrt{\gamma^2 + b^2}$ and then set $\ell= \sqrt{\gamma^2 + b^2}$. We can also upper bound $\| \exp(t H_v) \|$
\begin{align*}
\| \exp(t H_v) \| = \| \exp(t \cdot \text{Re}(H_v) ) \| = \left\| \exp \left( \gamma t \left( \frac{f(v)}{2} \otimes I_d + I_d \otimes \frac{f(v)}{2} \right) \right) \right\| \leq \exp(\gamma t).
\end{align*}

Note that $\sum_{v\in V} H_v=0$ since $\sum_{v\in V} f(v) =0$.

\begin{claim}\label{cla:bound_z_i_bot}
$\| z_i^{\bot} \| \leq \frac{\alpha_2}{1-\alpha_4} \max_{j<i} \| z_j^{\parallel} \|$. 
\end{claim}
\begin{proof}
\begin{align*}
\| z_i^{\bot} \| = & ~ \| ( E \widetilde{P} z_{i-1} )^{\bot} \| \\
\leq & ~ \| (E \widetilde{P} z_{i-1}^{\parallel} )^{\bot} \| + \| (E \widetilde{P} z_{i-1}^{\bot} )^{\bot} \| \\
\leq & ~ \alpha_2 \| z_{i-1}^{\parallel} \| + \alpha_4 \| z_{i-1}^{\bot} \| \\
\leq & ~ (\alpha_2 + \alpha_2 \alpha_4 + \alpha_2 \alpha_4 + \cdots ) \cdot \max_{j<i} \| z_j^{\parallel} \| \\
\leq & ~ \frac{\alpha_2}{1-\alpha_4} \max_{j < i} \| z_j^{\parallel} \|,
\end{align*}
where the first step follows by definition of $z_i$, the second step follows by triangle inequality, the third step follows by part 2 and 4 of Lemma~\ref{lemma:healy}.
\end{proof}

\begin{claim}\label{cla:bound_z_i_parallel}
$\| z_i^{\parallel} \| \leq (\alpha_1 + \frac{\alpha_2 \alpha_3}{1-\alpha_4}) \max_{j<i} \| z_j^{\parallel} \|$.
\end{claim}
\begin{proof}
\begin{align*}
\| z_i^{\parallel} \| = & ~ \| ( E \widetilde{P} z_{i-1} )^{\parallel} \| \\
\leq & ~ \| (E \widetilde{P} z_{i-1}^{\parallel} )^{\parallel} \| + \|  (E \widetilde{P} z_{i-1}^{\bot} )^{\parallel} \| \\
\leq & ~ \alpha_1 \| z_{i-1}^{\parallel} \| + \alpha_3 \| z_{i-1}^{\bot} \| \\
\leq & ~ \alpha_1 \| z_{i-1}^{\parallel} \| + \alpha_3 \frac{\alpha_2}{1-\alpha_4} \max_{j < i-1} \| z_j^{\parallel} \| \\
\leq & ~ (\alpha_1 + \frac{\alpha_2 \alpha_3}{1-\alpha_4} ),
 \end{align*}
where the first step follows by definition of $z_i$, the second step follows by triangle inequality, the third step follows by part 1 and 3 of Lemma~\ref{lemma:healy}, the fourth step follows by Claim~\ref{cla:bound_z_i_bot}.
\end{proof}

Combining Claim~\ref{cla:bound_z_i_bot} and Claim~\ref{cla:bound_z_i_parallel} gives 
\begin{align*}
\| z_k^{\parallel} \| \leq (\alpha_1 + \frac{\alpha_2 \alpha_3}{1-\alpha_4})^k  \| z_0^{\parallel} \| = \sqrt{d} \cdot (\alpha_1 + \frac{\alpha_2 \alpha_3}{1-\alpha_4})^k, 
\end{align*}
which implies that
\begin{align*}
\left \langle z_0 , (E \widetilde{P})^k z_0 \right \rangle \leq d \cdot (\alpha_1 + \frac{\alpha_2 \alpha_3}{1-\alpha_4})^k.
\end{align*}

Now the question is how to bound $(\alpha_1 + \frac{\alpha_2 \alpha_3}{1-\alpha_4})^k$.
We can upper bound $\alpha_1$, $\alpha_2 \alpha_3$ and $\alpha_4$ in the following sense,
\begin{align*}
\alpha_1 = \exp(t\ell) -t\ell \leq 1+ t^2 \ell^2 = 1+ t^2 (\gamma^2 +b^2) , 
\end{align*}
and
\begin{align*}
\alpha_2 \alpha_3 = \lambda (\exp(t\ell)-1)^2 \leq \lambda( 2t\ell  )^2  = 4 \lambda t^2 (\gamma^2 + b^2)
\end{align*}
where the second step follows by $t\ell < 1$ (because $\exp(x) \leq 1+2x, \forall x \in [0,1]$),
\begin{align*}
\alpha_4 = \lambda \cdot \exp(t\gamma) \leq \lambda (1 + 2t \gamma) \leq \frac{1}{2} + \frac{1}{2} \lambda
\end{align*}
where the second step follows by $t \gamma < 1$, and the third step follows by $t\gamma \leq (1-\lambda) /4\lambda$. 

Thus,
\begin{align*}
(\alpha_1+ \frac{\alpha_2 \cdot \alpha_3}{1-\alpha_4})^k \leq & ~ \left( 1+ t^2 (\gamma^2 +  b^2) + \frac{ 4 \lambda t^2 (\gamma^2 +b^2)}{\frac{1}{2} - \frac{1}{2}\lambda} \right)^k \\
\leq & ~ \exp \left( k t^2 (\gamma^2 +b^2) (1+ \frac{8}{1-\lambda}) \right).
\end{align*}

\end{proof}
\begin{remark} As is the case with Healy's proof \cite{healy08}, our proof also works for the case when there are different mean zero functions $f_1,\ldots, f_k$ for the different steps of the walk and also when there are $k$ $\lambda$-expanders $G_1,\ldots, G_k$ and the $j^{\text{th}}$ step of the walk is taken according to $G_j$. 
\end{remark}

\begin{remark} We suspect that with appropriate modifications, our proof should generalize to random walks on irregular undirected graphs (or reversible Markov chains) as was done for Healy's proof in \cite{ChungLLM12}.
\end{remark}

\begin{remark} Although we have stated the theorem for Hermitian matrices, the
	same result can be obtained for general matrices by a standard dilation
	trick, namely replacing every $d\times d$ matrix $M$ that appears with
	the $2d\times 2d$ Hermitian matrix $$\begin{bmatrix} 0 & M\\ M^* &
	0\end{bmatrix},$$ whose norm is always within a factor of two of
$M$.\end{remark}

 \section{Proof of Theorem \ref{thm:martingale}}\label{sec:martingale}

\begin{proof}Observe that for every $i=2,\ldots, k$ we have
$$ \E [ f(v_i) ~|~ v_{i-1} ] = \sum_{u\sim v_{i-1}}P(v_{i-1},u) f(u) = Pf(v_{i-1}),$$
whence the random vectors
$$ Y_i\up{1} := f(v_i)-Pf(v_{i-1})$$
satisfy
$$ \E [ Y_i\up{1} ~|~ v_1,\ldots,v_{i-1} ]=0$$
and thus form a martingale difference sequence with respect to the filtration generated by initial segments
of $v_1,\ldots, v_k$. Denoting $Y\up{1}_1:=f(v_1)$, 
we can write the sum of interest as a martingale part plus a remainder, which is
a sum of $k-1$ (i.e., one fewer) random variables:
$$ S = \sum_{i=1}^k Y\up{1}_i + \sum_{i=1}^{k-1}Pf(v_i).$$

Notice that $Pf$ is also a mean zero function on $G$, and by Jensen's inequality we have
$$ \|(Pf)(v)\|_*\le M:=\max_{v\in V}\|f(v)\|_*\qquad\textrm{for all $v$.}$$ The key point is that the
remainder terms $Pf(v_i)$ are smaller on average than the original terms $f(v_i)$ in 
squared Euclidean norm, because $P$ is a contraction orthogonal to the
constant vector; in particular, by considering the action of $P$ on each
coordinate of $f$ separately, we have:
\begin{equation}\label{eqn:shrink}\sum_v \|Pf(v)\|_2^2 \le
\lambda\cdot\sum_v \|f(v)\|_2^2.\end{equation}

Iterating this construction on the remainder a total of $T\le k$ times, we
obtain a sequence of martingales $1\le t\le T$:
$$Y\up{t}_1:=P^{t-1}f(v_1)\qquad Y\up{t}_i:= P^{t-1}f(v_i)-P^tf(v_{i-1}),\qquad
i=1\ldots, k-(t-1)$$
which are related to the original sum as:
$$ S = \sum_{t=1}^T \sum_{i=1}^{k-t+1}Y\up{t}_i + \sum_{i=1}^{k-T}P^Tf(v_i).$$
Interchanging the order of summation, we find that the random matrices
$$Z_i:= \sum_{t=1}^{\min\{(k+1-i), T\}}Y_i\up{t}$$
themselves form a martingale difference sequence, with each
$$\|Z_i\|_*\le \sum_t \|Y_i\up{t}\|_*\le TM.$$

We bound the error $W:=\frac1k \sum_{i=1}^{k-T} (P^Tf)(v_i)$ crudely as:
$$ \|W\|_2 \le \frac1k \sum_{i=1}^{k-T}\|P^Tf(v_i)\|_2 \le
\frac{k-T}{k}\sum_{v\in V} \|(P^Tf)(v)\|_2\le \lambda^{T/2}F\le
\exp(-(1-\lambda) T/2)F,$$
where $F:=\sqrt{\sum_{v\in V}\|f(v)\|_2^2}$, by applying \eqref{eqn:shrink}.
Rearranging and setting $T=2\log(F/\epsilon)/(1-\lambda)$ yields the advertised
bound on $\|Z_i\|_*$.
\end{proof}


\section{Acknowledgments}
We thank S\'{e}bastien Bubeck, Yash Deshpande, Ravi Kannan, Zeph Landau, Jelani Nelson, Eric Price, Aviad Rubinstein, Avi Wigderson, David P. Woodruff, David Xiao and David Zuckerman for helpful conversations, as well
Microsoft Research India and the Simons Institute for the Theory of Computing,
where much of this work was carried out.

\bibliographystyle{alpha}

\appendix

\section{Elementary calculations}

\begin{lemma}\label{lem:h_map}
We define function $h(z) : \C \rightarrow \C$ as follows
\begin{align*}
h(z) = - \frac{1+z}{1-z} + \sqrt{ \left( \frac{1+z}{1-z} \right)^2 + 1 }
\end{align*}
Then function $h(z)$ maps the unit disk $\{ z\in \C : |z| \leq 1 \}$ to the half disk $\{ z \in \C : |z| \leq 1 \mathrm{~and~} \mathrm{Re}(z) \geq 0 \}$.
\end{lemma}

\begin{proof}
We first compute the inverse of function $h(z)$, let $f = h^{-1}$. By definition of $h(z)$, we can do the following elementary calculations,
\begin{align*}
& ~ h(z) + \frac{1+z}{1-z} = \sqrt{ \left( \frac{1+z}{1-z} \right)^2 + 1 }\\
\implies & ~ \left(h(z) + \frac{1+z}{1-z}\right)^2 = \left( \frac{1+z}{1-z} \right)^2 + 1 \\
\implies & ~ (h(z))^2 + 2 h(z) \frac{1+z}{1-z}  = 1 \\
\implies & ~ (1-z) (h(z))^2+ 2 h(z) (1+z) = (1-z) \\
\implies & ~ (h(z))^2 + 2 h(z) - 1 = z( (h(z))^2 - 2h(z) -1 ),
\end{align*}
Thus, we obtain that
\begin{align*}
f(z) = \frac{z^2 +  2z -1}{z^2 - 2z -1}
\end{align*}
To finish the proof, we need Claim~\ref{cla:z_inside_half_disk} and Claim~\ref{cla:z_on_boundary_of_half_disk}.

\begin{claim}\label{cla:z_inside_half_disk}
For all $z \in \{ z \in \C : |z| \leq 1 \mathrm{~and~} \mathrm{Re}(z) \geq 0 \}$,
\begin{align*}
|f(z) | \leq 1
\end{align*}
\end{claim}
\begin{proof}
Let $z = r e^{\i \phi}$, where $r \in [0,1]$ and $\phi \in [-\pi/2,\pi/2]$. It is easy to observe that $\cos(\phi) \in [0,1]$ We have
\begin{align*}
|f(z)| = & ~ \left| \frac{z^2 + 2 z - 1}{z^2 - 2 z-1} \right| \\
= & ~ \left| \frac{r^2 e^{\i 2\phi} + 2 r e^{\i \phi} - 1 }{ r^2 e^{\i 2\phi} -2 r e^{\i \phi} -1  } \right| \\
= & ~ \frac{ | r^2 e^{\i 2\phi} + 2 r e^{\i \phi} - 1 | }{ | r^2 e^{\i 2\phi} -2 r e^{\i \phi} -1 | } 
\end{align*}

We can compute the numerator,
\begin{align*}
 | r^2 e^{\i 2 \phi} + 2 r e^{\i \phi} - 1 |^2
= & ~ | r^2 \cos 2\phi + \i r^2 \sin 2\phi  + 2 r \cos \phi + 2 r \i \sin \phi - 1 |^2 \\
= & ~   (r^2 \cos 2\phi + 2 r \cos \phi - 1)^2 + (r^2 \sin 2\phi + 2 r \sin \phi)^2  \\
= & ~   r^4 + 4r^2 + 1 - 2 r^2 \cos 2\phi + 4 r^3 \cos 2\phi \cos \phi - 4 r \cos \phi + 4 r^3 \sin2\phi \sin \phi   \\
= & ~   r^4 + 4r^2 + 1 - 2 r^2 \cos 2\phi + 4r^3 \cos \phi - 4 r \cos \phi.
\end{align*}
We can compute the denominator,
\begin{align*}
 |e^{\i 2 \phi} - 2 e^{\i \phi} - 1|^2
= & ~ | r^2 \cos 2\phi + \i r^2 \sin 2\phi  - 2 r \cos \phi - 2 r \i \sin \phi - 1 |^2 \\
= & ~   (r^2 \cos 2\phi - 2 r \cos \phi - 1)^2 + (r^2 \sin 2\phi - 2 r \sin \phi)^2  \\
= & ~   r^4 + 4r^2 + 1 - 2 r^2 \cos 2\phi - 4 r^3 \cos 2\phi \cos \phi + 4 r \cos \phi - 4 r^3 \sin2\phi \sin \phi   \\
= & ~   r^4 + 4r^2 + 1 - 2 r^2 \cos 2\phi - 4r^3 \cos \phi + 4 r \cos \phi.
\end{align*}
Note that, in order to show $|f(z)| \leq 1$, it is sufficient to prove
\begin{align*}
4 r^3 \cos \phi - 4 r\cos \phi \leq -4 r^3 \cos\phi + 4 r \cos \phi
\end{align*}
which is equivalent to
\begin{align*}
8r(1-r^2) \cos \phi \geq 0.
\end{align*}
It follows by definition of $r$ and $\phi$. Thus, we complete the proof.
\end{proof}

Next, we can show that
\begin{claim}\label{cla:z_on_boundary_of_half_disk}
For all $z$ is on the boundary of half disk, 
\begin{align*}
|f(z)| = 1.
\end{align*}
\end{claim}
\begin{proof}
First, we want to show that $\forall z \in [-\i,\i]$, $|f(z)| = 1$. Let $b \in [0,1]$, let $z= \i b$, then we have
\begin{align*}
|f(z)| = & ~ | f(\i b) | \\
= & ~ \left| \frac{-b^2 + 2b\i -1}{-b^2 - 2b \i -1} \right|\\
= & ~ 1.
\end{align*}
Second, we want to show that for all $z$ on half circle, $|f(z)| = 1$.
We replace $z$ by $e^{\i \phi}$, where $\phi \in [-\pi/2,\pi/2]$. Then we have
\begin{align*}
| f(z) | = & ~ \left| \frac{z^2+2z-1}{z^2 - 2 z -1} \right| \\
= & ~ \left| \frac{ e^{\i 2 \phi} + 2 e^{\i \phi} - 1 }{  e^{\i 2 \phi} - 2 e^{\i \phi} - 1 } \right|  \\ 
= & ~ \frac{ | e^{\i 2 \phi} + 2 e^{\i \phi} - 1 | }{ | e^{\i 2 \phi} - 2 e^{\i \phi} - 1  | } .
\end{align*}
We can compute the numerator,
\begin{align*}
 | e^{\i 2 \phi} + 2 e^{\i \phi} - 1 |
= & ~ | \cos 2\phi + \i \sin 2\phi  + 2  \cos \phi + 2\i \sin \phi - 1 | \\
= & ~ \left( (\cos 2\phi + 2 \cos \phi - 1)^2 + (\sin 2\phi + 2\sin \phi)^2 \right)^{1/2} \\
= & ~ \left( 6 - 2 \cos 2\phi + 4 \cos 2\phi \cos \phi - 4 \cos \phi + 4 \sin2\phi \sin \phi \right)^{1/2} \\
= & ~ \left( 6 - 2 \cos 2\phi + 4 \cos \phi - 4 \cos \phi \right)^{1/2} \\
= & ~ \left( 6 - 2 \cos 2\phi  \right)^{1/2}.
\end{align*}
We can compute the denominator,
\begin{align*}
 |e^{\i 2 \phi} - 2 e^{\i \phi} - 1| 
= & ~ | \cos 2\phi + \i \sin 2\phi  - 2  \cos \phi - 2\i \sin \phi - 1 | \\
= & ~ \left( (\cos 2\phi - 2 \cos \phi - 1)^2 + (\sin 2\phi - 2\sin \phi)^2 \right)^{1/2} \\
= & ~ \left( 6 - 2 \cos 2\phi - 4 \cos 2\phi \cos \phi + 4 \cos \phi - 4 \sin2\phi \sin \phi \right)^{1/2} \\
= & ~ \left( 6 - 2 \cos 2\phi + 4 \cos \phi - 4 \cos \phi \right)^{1/2} \\
= & ~ \left( 6 - 2 \cos 2\phi  \right)^{1/2}.
\end{align*}
 
Thus, we have
\begin{align*}
|f(z)| = 1
\end{align*}

\end{proof}
Note the biholomorphic is basically follows from the formula of $f$ and $h$, because they are composition of holomorphic function.
\end{proof}


\begin{lemma}\label{lem:theta_square}
For any $\rho \in [0,1]$ and $\cos \varphi \in [-1,0]$, we have
\begin{align*}
\frac{1-\rho}{1-\cos \varphi} - (1-\rho)^2 \leq \frac{1-\rho^2}{1-2\rho \cos \varphi + \rho^2} \leq \frac{1-\rho}{1-\cos \varphi} + 2(1-\rho)^2.
\end{align*}
\end{lemma}
\begin{proof}
This directly follows by combining Claim~\ref{cla:theta_square_1} and Claim~\ref{cla:theta_square_2}
\end{proof}

\begin{claim}\label{cla:theta_square_1}
There exists some sufficiently large constant $c\geq 1$ such that for any $\rho \in [0,1]$ and $\cos\varphi \in [-1,0]$, we have
\begin{align*}
\frac{1-\rho^2}{1-2\rho \cos \varphi + \rho^2} \leq \frac{1-\rho}{1-\cos \varphi} + c(1-\rho)^2.
\end{align*}
\end{claim}
\begin{proof}

It is equivalent to
\begin{align*}
\frac{1+\rho}{1-2\rho \cos \varphi + \rho^2} \leq & ~ \frac{1}{1-\cos \varphi} + c(1+\rho) \\
(1+\rho) (1-\cos \varphi) \leq & ~ 1-2\rho \cos \varphi + \rho^2 + c (1+\rho) (1-\cos\varphi) (1-2\rho \cos \varphi + \rho^2) \\
\rho - \cos \varphi \leq & ~ -\rho \cos \varphi + \rho^2 + c (1+\rho) (1-\cos\varphi) (1-2\rho \cos \varphi + \rho^2)
\end{align*}
which is equivalent to,
\begin{align*}
- \rho \cos \varphi + \rho^2 - \rho + \cos \varphi + c (1+\rho)(1- \cos \varphi) (1-2\rho \cos \phi + \rho^2) \geq 0
\end{align*}

Since $1-2\rho \cos\varphi + \rho^2 \geq 1$, thus it suffices to show
\begin{align*}
(\rho - 1 ) (\rho - \cos \varphi) + c (1+\rho) (1-\cos\varphi) \geq 0
\end{align*}
Note that $(\rho - 1) (\rho - \cos \varphi) \geq -2$, by choosing $c\geq 2$, we complete the proof.
\end{proof}

\begin{claim}\label{cla:theta_square_2}
There exists some sufficiently large constant $c\geq 1$ such that for any $\rho \in [0,1]$ and $\cos\varphi \in [-1,0]$, we have
\begin{align*}
\frac{1-\rho^2}{1-2\rho \cos \varphi + \rho^2} \geq \frac{1-\rho}{1-\cos \varphi} - c(1-\rho)^2.
\end{align*}
\end{claim}
\begin{proof}

It is equivalent to
\begin{align*}
\frac{1+\rho}{1-2\rho \cos \varphi + \rho^2} \geq & ~ \frac{1}{1-\cos \varphi} + c(1+\rho) \\
(1+\rho) (1-\cos \varphi) \geq & ~ 1-2\rho \cos \varphi + \rho^2 - c (1+\rho) (1-\cos\varphi) (1-2\rho \cos \varphi + \rho^2) \\
\rho - \cos \varphi \geq & ~ -\rho \cos \varphi + \rho^2 - c (1+\rho) (1-\cos\varphi) (1-2\rho \cos \varphi + \rho^2)
\end{align*}
which is equivalent to,
\begin{align*}
- \rho \cos \varphi + \rho^2 - \rho + \cos \varphi - c (1+\rho)(1- \cos \varphi) (1-2\rho \cos \phi + \rho^2) \leq 0
\end{align*}
which is equivalent to
\begin{align*}
(\rho-1) (\rho -\cos\varphi) \leq c(1+\rho) (1-\cos\varphi) (1-2\rho \cos\varphi + \rho^2)
\end{align*}
It suffices to choose $c=1$
\end{proof}

\end{document}